\documentclass[a4paper,fleqn,10pt]{article}
\usepackage{hyperref,amsthm,amsmath,amssymb,amsfonts,color}
\newtheorem{theorem}{Theorem}

\numberwithin{equation}{section}
\numberwithin{lemma}{section}
\numberwithin{theorem}{section}
\numberwithin{corollary}{section}

\allowdisplaybreaks

\begin{document}
\title{On the discrete analogues of Appell function $F_4$}
\author{Ravi Dwivedi$^{1,}$\footnote{E-mail: dwivedir999@gmail.com}   \, and Vivek Sahai$^{2,}$\footnote{E-mail: sahai\_vivek@hotmail.com (Corresponding author)} \\ ${}^{1}$Department of Science, SAGEMMC, Jagdalpur, Bastar, CG, 494001, India; \\ ${}^{2}$Department of Mathematics and Astronomy, Lucknow University, \\ Lucknow 226007, India.}
\maketitle 
\begin{abstract}
	In this paper, we study the Appell function $F_4$ from discrete point of view. In particular, we obtain regions of convergence, difference-differential equations, finite and infinite summation formulas and a list of recursion relations satisfied by the discrete analogues of Appell function $F_4$. 
	
	
\end{abstract}
\section{Introduction}
	This is the fourth and final paper in a series of papers dealing with the discrete analogues of Appell functions. In the previous three papers \cite{ds16, ds17, ds18}, we have studied various properties of discrete Appell functions $\mathcal{F}^{(i)}_1$, $\mathcal{F}^{(i)}_2$ and $\mathcal{F}^{(i)}_3$, $i = 1, 2$ as well as discrete analogues of seven Humbert functions. Appell function $F_4$ is the last in the list of four Appell functions. The analysis of Appell function $F_4$ is not as smooth as that of Appell functions $F_1$, $F_2$ and $F_3$. We can see that the Appell function $F_4$ has no degenerations and simple integral representations. Nevertheless, its appearance in several applications make the Appell function $F_4$  notable. This results in the availability of a wide variety of literature on $F_4$ from different points of view \cite{bts, ds1, ds2, ds5, am, nn}. The series representation of $F_4$ is given by, \cite{kdf, emo} 
\begin{align}
		F_4 (a, b; c, c'; x, y) & = \sum_{m, n = 0}^{\infty} \frac{(a)_{m + n} \, (b)_{m + n}  }{(c)_{m} \, {(c')}_n} \, \frac{x^m \, y^n}{m ! \, n!}, \quad \sqrt{\vert x \vert} +  \sqrt{\vert y \vert} < 1,
\end{align}
where the shifted factorial $(a)_l$ is defined by
\begin{equation}
	(a)_l = \begin{cases}
		1, & \text{ if } l = 0,\\
		a \, (a + 1) \cdots (a + l - 1), & \text{ if } l \ge 1.
	\end{cases}
\end{equation}

In this paper, we study two distinct discrete analogues of Appell function $F_4$ and investigate its fundamental properties. Let  $a$, $b$, $c_1$, $c_2$, $t$, $t_1$ and $t_2$ be complex numbers such that $\Re (c_1), \Re (c_2) \ne 0, -1, -2, \dots$. For $k, k_1, k_2 \in \mathbb{N}$, we define the discrete analogues of Appell hypergeometric  function $\mathcal{F}_4$ as follows: 
\begin{align}
	\mathcal{F}^{(1)}_4 & = \mathcal{F}^{(1)}_4 (a, b; c_1, c_2; t_1, t_2, k_1, k_2, x, y)\nonumber\\ 
	& = \sum_{m,n\geq0} \frac{(a)_{m+n} \, (b)_{m + n} \,  (-1)^{m k_1} \, (-t_1)_{mk_1} \, (-1)^{nk_2} \, (-t_2)_{nk_2}}{ (c_1)_{m} \, (c_2)_n \, m! \, n!} \ x^m \, y^n; \label{3.1}
\end{align} 
\begin{align}
	\mathcal{F}^{(2)}_4 & =	\mathcal{F}^{(2)}_4 (a, b; c_1, c_2; t, k, x, y)\nonumber\\
	& = \sum_{m,n\geq0} \frac{(a)_{m+n} \, (b)_{m + n} \, (-1)^{(m + n) k} \, (-t)_{(m + n)k}}{(c_1)_{m} \, (c_2)_n \, m! \, n!} \ x^m \, y^n. \label{3.2}
\end{align}
		
	The paper is organized as follows. In Section~2, we discuss some special cases of first discrete version of  Appell function $\mathcal{F}_4^{(1)}$ introduced in \eqref{3.1}, its region of convergence and difference-differential equations obeyed by it.  We also obtain some integral representations for the  discrete Appell function $\mathcal{F}_4^{(1)}$. In Section~3, differential and difference formulae for discrete Appell function $\mathcal{F}_4^{(1)}$ are determined. In Section~5, we give the recursion formulae satisfied by $\mathcal{F}_4^{(1)}$. We also present a list of first and second order recursion relations of  $\mathcal{F}_4^{(1)}$. In Section~6, the second discrete form of Appell function $F_4$ \emph{viz.} $\mathcal{F}_4^{(2)}$, is studied. We  present the results and theorems for $\mathcal{F}_4^{(2)}$ directly without giving elaborate proofs. 
\section{Discrete Appell function $\mathcal{F}^{(1)}_4$}	  
     Here, and in subsequent sections, we explore the discrete Appell  function $\mathcal{F}^{(1)}_4$ and later in last section, we list the results for second discrete  Appell function $\mathcal{F}^{(2)}_4$. 
     
	  For some fixed values of $k_1$ and $k_2$, the discrete Appell function $\mathcal{F}^{(1)}_4$ reduces into some classical functions of two variables. In particular, we have
	  \begin{align}
	  	\mathcal{F}^{(1)}_4(a, b; c_1, c_2; t_1, t_2, 0, 0, x, y) & = 
	  	 F_4 (a, b; c_1, c_2; x, y),
	  \end{align}
      \begin{align}
     	\mathcal{F}^{(1)}_4(a, b; c_1, c_2; t_1, t_2, 1, 0, x, y)
     & = F_{{0}:{1}, {1}} ^{{2}:{1}, {0}}\left(\begin{array}{ccc}
      	a, b: & -t_1, & -\\
      	-: & c_1, & c_2 
      \end{array}; -x, \ y\right),
      \end{align}
  \begin{align}
  	\mathcal{F}^{(1)}_4(a, b; c_1, c_2; t_1, t_2, 0, 1, x, y)
  	& = F_{{0}:{1}, {1}} ^{{2}:{0}, {1}}\left(\begin{array}{ccc}
  		a, b: & -, & -t_2\\
  		-: & c_1, & c_2 
  	\end{array}; x, \ -y\right),
  \end{align}
\begin{align}
	 \mathcal{F}^{(1)}_4(a, b; c_1, c_2; t_1, t_2, 1, 1, x, y)
	& = F_{{0}:{1}, {1}} ^{{2}:{1}, {1}}\left(\begin{array}{ccc}
		a, b: & -t_1, & -t_2\\
		-: & c_1, & c_2 
	\end{array}; -x, \ -y\right), 
\end{align}
where $F_{{l'}:{m'}, {n'}} ^{{l}:{m}, {n}}$ is the Kamp\'e de F\'eriet function defined by, \cite{emo}
\begin{align}
		& F_{{l'}:{m'}, {n'}} ^{{l}:{m}, {n}} \left(\begin{array}{ccc}
				A: & B, & C\\
				D: &E, &F 
			\end{array}; x, \ y\right)\nonumber\\
		& =\displaystyle  \sum_{k,p\geq 0} \ \frac{\prod_{i=1}^{l} (a_i)_{k+p} \, \prod_{i=1}^{m} (b_i)_{k} \, \prod_{i=1}^{n} (c_i)_{p}}{\prod_{i=1}^{l'} (d_i)_{k+p}   \, \prod_{i=1}^{m'} (e_i)_{k}  \, \prod_{i=1}^{n'} (f_i)_{p}} \ \frac{x^k \, y^p}{k! \, p!},\label{c1eq71}
	\end{align}
	where $A$ denote the sequence of complex numbers $a_1$, \dots, $a_{l}$, etc. and $d_i, e_i, f_i \ne 0, -1, \dots$.
	
Now, we examine the convergence of the discrete Appell function $\mathcal{F}^{(1)}_4$. Let $\mathcal{A}_{m, n} x^m \, y^n$ be the general term of  $\mathcal{F}^{(1)}_4$. Then
  \begin{align}
  &	\left \vert \mathcal{A}_{m, n} x^m \, y^n \right \vert\nonumber\\
   & = \left \vert \frac{(a)_{m+n} \, (b)_{m + n} \, (-t_1)_{mk_1} \, (-t_2)_{nk_2}}{ (c_1)_{m} \, (c_2)_n \, m! \, n!} \ x^m \, y^n \right \vert\nonumber\\
  	& < \left \vert \frac{\Gamma (c_1) \, \Gamma (c_2)}{\Gamma (a) \, \Gamma (b)  \, \Gamma (-t_1) \, \Gamma (-t_2)} \right \vert \nonumber\\
  	& \quad \times \left \vert \frac{\Gamma (a + m + n) \, \Gamma (b + m + n) \, \Gamma (-t_1 + mk_1) \, \Gamma (-t_2 + nk_2)}{\Gamma (c_1 + m) \, \Gamma (c_2 + n) \, m! \, n!} \right \vert \, \vert x \vert^m \, \vert y\vert^n.
  \end{align}
For sufficiently large $m$ and $n$, the limiting value of Stirling formula $\lim_{n \to \infty} \Gamma (\lambda + n) = \sqrt (2 \pi) \, n^{\lambda + n - \frac{1}{2}} \, e^{-n}$ yields 
\begin{align}
	&	\left \vert \mathcal{A}_{m, n} x^m \, y^n \right \vert\nonumber\\
	& < \left \vert \frac{2 \pi \, \Gamma (c_1) \, \Gamma (c_2)}{\Gamma (a) \, \Gamma (b)  \, \Gamma (-t_1) \, \Gamma (-t_2)} \right \vert \nonumber\\
	& \quad \times \left \vert  (m + n)^{a + b - 2} \, m^{1 - c_1} \, n^{1 - c_2} \, (mk_1)^{mk_1 - t_1 - \frac{1}{2}} \, (nk_2)^{nk_2 - t_2 - \frac{1}{2}} e^{- (mk_1 + nk_2) } \right \vert \nonumber\\
	& \quad \times  \left\vert \left(\frac{(m + n)!}{m ! \, n!}\right)^2 x^m \, y^n\right\vert.
\end{align}
Let $N > \left \vert \frac{2 \pi \, \Gamma (c_1) \, \Gamma (c_2)}{\Gamma (a) \, \Gamma (b)  \, \Gamma (-t_1) \, \Gamma (-t_2)} \right \vert$. Then
\begin{align}
		\left \vert \mathcal{A}_{m, n} x^m \, y^n \right \vert & < \frac{N \, (mk_1)^{mk_1 - t_1 - \frac{1}{2}} \, (nk_2)^{nk_2 - t_2 - \frac{1}{2}}}{ (m + n)^{2 - a - b} \, m^{c_1 - 1} \, n^{c_2 - 1} \, e^{(mk_1 + nk_2) }} \, \left\vert \left(\frac{(m + n)!}{m ! \, n!}\right)^2 x^m \, y^n\right\vert\nonumber\\
		& < \frac{N \, (mk_1)^{mk_1 - t_1 - \frac{1}{2}} \, (nk_2)^{nk_2 - t_2 - \frac{1}{2}}}{ (m + n)^{2 - a - b} \, m^{c_1 - 1} \, n^{c_2 - 1} \, e^{(mk_1 + nk_2) }} \, (\sqrt{\vert x\vert} + \sqrt{\vert y \vert})^{2(m + n)}.
\end{align}
For $k_1, k_2 \in \mathbb{N}, t_1, t_2 \in \mathbb{C}$ and $\sqrt{\vert x\vert} + \sqrt{\vert y \vert} < 1$, $\left \vert \mathcal{A}_{m, n} x^m \, y^n \right \vert \to 0$ as $m, n \to \infty$. Hence the discrete Appell function $\mathcal{F}^{(1)}_4$ converges absolutely. 

 Let $\Theta_t : = t \, \rho_t \, \Delta_t$ be the discrete analogue of $\theta = t \, \frac{d}{dt}$, where $\Delta_t f(t) = f(t + 1) - f(t)$ and $\rho_t \, f (t) = f(t - 1)$. Then, we have
\begin{align}
\Theta_t \, ((-1)^{nk} \, (-t)_{nk}) & = n\, k \, (-1)^{nk} \, (-t)_{nk}. 
\end{align} 
Next, we obtain the difference equations satisfied by $\mathcal{F}^{(1)}_4$. We have
\begin{align}
	& \Theta_{t_1} \left(\frac{1}{k_1} \Theta_{t_1} + c_1 - 1\right) \, \mathcal{F}^{(1)}_4\nonumber \\
	& = \sum_{m,n \geq 0} \frac{(a)_{m+n} \, (b)_{m + n} \, (-1)^{m k_1} \, (-t_1)_{mk_1} \, (-1)^{nk_2} \, (-t_2)_{nk_2}}{ (c_1)_{m} \, (c_2)_n \, m! \, n!} \ x^m \, y^n \, mk_1 \, (c_1 + m - 1)\nonumber\\
	& = k_1 \, \sum_{m, n \ge 0} \frac{(a)_{m+n + 1} \, (b)_{m + n + 1} \, (-1)^{(m  + 1) k_1} \, (-t_1)_{(m + 1)k_1} \, (-1)^{n k_2} \, (-t_2)_{nk_2}}{ (c_1)_{m} \, (c_2)_n \, m! \, n!} \ x^{m + 1} \, y^n\nonumber\\
	& = k_1 \, \sum_{m, n \ge 0} (a + m + n) \, (b + m + n) \, (-1)^{k_1} \, (-t_1)_{k_1} \, x \, \rho_{t_1}^{k_1} \nonumber\\
	& \quad \times \frac{(a)_{m+n} \, (b)_{m + n} \, (-1)^{m k_1} \, (-t_1+k_1)_{m k_1} \, (-1)^{nk_2} \, (-t_2)_{nk_2}}{(c_1)_{m} \, (c_2)_n \, m! \, n!} \ x^{m} \, y^n\nonumber\\
	& = k_1 \, (-1)^{k_1} \, (-t_1)_{k_1} \, x \, \rho_{t_1}^{k_1} \,  \left(\frac{1}{k_1} \Theta_{t_1} + \frac{1}{k_2} \Theta_{t_2} + a\right) \, \left(\frac{1}{k_1} \Theta_{t_1} + \frac{1}{k_2} \Theta_{t_2}  + b\right) \, \mathcal{F}^{(1)}_4.
\end{align}
Thus, we arrive at
\begin{align}
&	\left[\Theta_{t_1} \left(\frac{1}{k_1} \Theta_{t_1} + c_1 - 1\right)\right. \nonumber\\
& \quad \left. -  k_1 \, (-1)^{k_1} \, (-t_1)_{k_1} \, x \, \rho_{t_1}^{k_1} \,  \left(\frac{1}{k_1} \Theta_{t_1} + \frac{1}{k_2} \Theta_{t_2} + a\right) \, \left(\frac{1}{k_1} \Theta_{t_1} + \frac{1}{k_2} \Theta_{t_2}  + b\right) \right] \mathcal{F}^{(1)}_4 = 0.\label{1.15}
\end{align}	    
	    Similarly
\begin{align}
&	\left[\Theta_{t_2} \left(\frac{1}{k_2} \Theta_{t_2} + c_2 - 1\right) \right. \nonumber\\
& \quad \left. -  k_2 \, (-1)^{k_2} \, (-t_2)_{k_2} \, y \, \rho_{t_2}^{k_2} \,  \left(\frac{1}{k_1} \Theta_{t_1} + \frac{1}{k_2} \Theta_{t_2} + a\right) \, \left(\frac{1}{k_1} \Theta_{t_1} + \frac{1}{k_2} \Theta_{t_2}  + b\right) \right] \mathcal{F}^{(1)}_4 = 0. \label{1.16}
\end{align}	    
%

We now give some of the integral representations of the discrete Appell  function $\mathcal{F}^{(1)}_4$.
\begin{align}
& \mathcal{F}^{(1)}_4(a, b; c_1, c_2; t_1, t_2, k_1, k_2, x, y)\nonumber\\
& = \frac{1}{\Gamma (a)} \int_{0}^{\infty} e^{-u} \, u^{a - 1} \nonumber\\
& \quad \times F_{{0}:{1}; {1}} ^{{1}:{k}; {k}}\left(\begin{array}{ccc}
	b : &  \frac{- t_1}{k}, \dots, \frac{- t_1 + k - 1}{k} ; &  \frac{- t_2}{k}, \dots, \frac{- t_2 + k - 1}{k}\\
	: & c_1 ; & c_2
\end{array}; (-k)^k \, u x,  (-k)^k \, u y\right) du;\label{3.9}\\
& = \frac{1}{\Gamma (b)} \int_{0}^{\infty} e^{-u} \, u^{b - 1} \nonumber\\
& \quad \times F_{{0}:{1}; {1}} ^{{1}:{k}; {k}}\left(\begin{array}{ccc}
	a : & \frac{- t_1}{k}, \dots, \frac{- t_1 + k - 1}{k} ; &  \frac{- t_2}{k}, \dots, \frac{- t_2 + k - 1}{k}\\
	-: & c_1 ; & c_2 
\end{array}; (-k)^k \, u x,  (-k)^k \, u y\right) du.\label{2.14}
\end{align}
To prove \eqref{3.9}, we use the identity $(a)_{m + n} = \frac{\Gamma (a + m + n)}{\Gamma (a)}$ and the integral of $\Gamma (a + m + n)$ as
\begin{align}
\Gamma (a + m + n) = \int_{0}^{\infty} e^{-u} \, u^{a + m + n - 1} \, du.
\end{align}
Similarly, the other shifted factorial of numerator can be used to get the  integral \eqref{2.14}.  
\section{Differential and difference formulae}
\begin{theorem} The following difference and differential formulas are satisfied by discrete Appell function $\mathcal{F}^{(1)}_4$, where $\theta = x \frac{\partial}{\partial x}, \phi = y \frac{\partial}{\partial y}$:
\begin{align}
&	(\Delta_{t_1})^r \mathcal{F}^{(1)}_4(a, b; c_1, c_2; t_1, t_2, 1, k_2, x, y) \nonumber\\
& = \frac{(a)_r \, (b)_r \,  x^r }{(c_1)_r}  \mathcal{F}^{(1)}_4(a + r, b + r; c_1 + r, c_2; t_1, t_2, 1, k_2, x, y);\label{4.1}\\
& (\Delta_{t_2})^r \mathcal{F}^{(1)}_4(a, b; c_1, c_2; t_1, t_2, k_1, 1, x, y) \nonumber\\
& = \frac{(a)_r \, (b)_r \,  y^r }{(c_2)_r}   \mathcal{F}^{(1)}_4(a + r, b + r; c_1, c_2 + r; t_1, t_2, k_1, 1, x, y);\label{4.2}\\
& (\theta)^r \mathcal{F}^{(1)}_4(a, b; c_1, c_2; t_1, t_2, k_1, k_2, x, y) \nonumber\\
& = \frac{(-1)^{rk_1} \, (a)_r \, (b)_r \, (-t_1)_{rk_1} \, x^r}{(c_1)_r} \, \mathcal{F}^{(1)}_4(a + r, b + r; c_1 + r, c_2; t_1 - rk_1, t_2, k_1, k_2, x, y);\\
& (\phi)^r \mathcal{F}^{(1)}_4(a, b; c_1, c_2; t_1, t_2, k_1, k_2, x, y) \nonumber\\
& = \frac{(-1)^{rk_2} \, (a)_r \, (b)_r \, (-t_2)_{rk_2} \, y^r}{(c_2)_r} \, \mathcal{F}^{(1)}_4 (a + r, b + r; c_1, c_2 + r; t_1, t_2 - rk_2, k_1, k_2, x, y).
\end{align}
\end{theorem}
\begin{proof}
The action of difference operator $\Delta_{t_1}$  on the discrete Appell function $\mathcal{F}^{(1)}_4$, when $k_1 = 1$, yields
 \begin{align}
 &(\Delta_{t_1}) \mathcal{F}^{(1)}_4(a, b; c_1, c_2; t_1, t_2, 1, k_2, x, y)\nonumber\\
& = \sum_{m,n \geq 0} \frac{(a)_{m+n} \, (b)_{m + n} \, (-1)^{m - 1} \, m \, (-t_1)_{m - 1} \, (-1)^{nk_2} \, (-t_2)_{nk_2}}{ (c_1)_{m} \, (c_2)_n \, m! \, n!} \ x^m \, y^n\nonumber\\
& = \sum_{m, n \geq 0} \frac{(a)_{m+n + 1} \, (b)_{m + n + 1} \, (-1)^{ m }  \, (-t_1)_{m} \, (-1)^{nk_2} (-t_2)_{nk_2}}{(c_1)_{m + 1} \, (c_2)_n \, m! \, n!} \ x^{m + 1} \, y^n\nonumber\\
& =   \frac{a \, b \, x}{c_1} \sum_{m,n \geq 0} \frac{(a + 1)_{m+n} \, (b + 1)_{m + n} \, (-1)^{m} \, (-t_1 )_{m} \, (-1)^{nk_2} \, (-t_2)_{nk_2}}{(c_1 + 1)_{m} \, (c_2)_n \, m! \, n!} \ x^m \, y^n\nonumber\\
& =   \frac{a \, b \, x}{c_1} \mathcal{F}^{(1)}_4 (a + 1, b + 1; c_1 + 1, c_2; t_1, t_2, 1, k_2, x, y).
 \end{align}
Now, acting the difference operator twice gives
  \begin{align}
  	&(\Delta_{t_1})^2 \mathcal{F}^{(1)}_4 (a, b; c_1, c_2; t_1, t_2, 1, k_2, x, y)\nonumber\\
  	& =   \frac{(a)_2 \, (b)_2 \, x^2}{(c_1)_2}  \mathcal{F}^{(1)}_4(a + 2, b + 2; c_1 + 2, c_2; t_1 , t_2, 1, k_2, x, y).
  \end{align}
In general,
\begin{align}
	&	(\Delta_{t_1})^r \mathcal{F}^{(1)}_4 (a, b; c_1, c_2; t_1, t_2, 1, k_2, x, y) \nonumber\\
	& = \frac{(a)_r \, (b)_r \, x^r }{(c_1)_r}   \mathcal{F}^{(1)}_4(a + r, b + r; c_1 + r, c_2; t_1, t_2, 1, k_2, x, y).
\end{align}
It completes the proof of \eqref{4.1}. Similarly the other difference and differential formulae can be proved. 
\end{proof}
Besides, some other differential formulas are stated in the following theorem.
\begin{theorem} 
The following differential formulas hold for the discrete Appell function $\mathcal{F}^{(1)}_4$:
\begin{align}
	& \left(\frac{\partial}{\partial x}\right)^r \left[x^{b + r - 1} \mathcal{F}^{(1)}_4(a, b; c_1, c_2; t_1, t_2, k_1, k_2, x, x y)\right]\nonumber\\
	& = x^{b - 1} \, (b)_r \mathcal{F}^{(1)}_4(a, b + r; c_1, c_2; t_1, t_2, k_1, k_2, x, x \, y);\label{4.14}\\
		& \left(\frac{\partial}{\partial y}\right)^r [y^{b + r - 1} \mathcal{F}^{(1)}_4(a, b; c_1, c_2; t_1, t_2, k_1, k_2, x \, y, y)]\nonumber\\
	& = y^{b - 1} \, (b)_r \mathcal{F}^{(1)}_4 (a, b + r; c_1, c_2; t_1, t_2, k_1, k_2, x \, y, y);\\
		& \left(\frac{\partial}{\partial x}\right)^r [x^{a + r - 1} \mathcal{F}^{(1)}_4(a, b; c_1, c_2; t_1, t_2, k_1, k_2, x, xy)]\nonumber\\
	& = x^{a - 1} \, (a)_r \mathcal{F}^{(1)}_4 (a + r, b; c_1, c_2; t_1, t_2, k_1, k_2, x, xy);\\
		& \left(\frac{\partial}{\partial y}\right)^r [y^{a + r - 1} \mathcal{F}^{(1)}_4 (a, b; c_1, c_2; t_1, t_2, k_1, k_2, xy, y)]\nonumber\\
	& = y^{a - 1} \, (a)_r \mathcal{F}^{(1)}_4(a + r, b; c_1, c_2; t_1, t_2, k_1, k_2, xy, y);\\
		& \left(\frac{\partial}{\partial x}\right)^r [x^{c_1 - 1} \mathcal{F}^{(1)}_4(a, b; c_1, c_2; t_1, t_2, k_1, k_2, x, y)]\nonumber\\
	& = (-1)^r \, (1 - c_1)_r \, x^{c_1 - r - 1} \mathcal{F}^{(1)}_4(a, b; c_1 - r, c_2; t_1, t_2, k_1, k_2, x, y);\\
		& \left(\frac{\partial}{\partial y}\right)^r [y^{c_2 - 1} \mathcal{F}^{(1)}_4(a, b; c_1, c_2; t_1, t_2, k_1, k_2, x, y)]\nonumber\\
	& = (-1)^r \, y^{c_2 - r - 1} \, (1 - c_2)_r \mathcal{F}^{(1)}_4 (a, b; c_1, c_2 - r; t_1, t_2, k_1, k_2, x, y).
\end{align}
\end{theorem}
\begin{proof}
To prove \eqref{4.14}, we start with 
\begin{align}
	& \frac{\partial}{\partial x} \left[x^{b} \mathcal{F}^{(1)}_4 (a, b; c_1, c_2; t_1, t_2, k_1, k_2, x, x \, y)\right]\nonumber\\
	& = \sum_{m, n \geq 0} \frac{(a)_{m+n} \, (b)_{m + n} \, (-1)^{m k_1} \, (-t_1)_{mk_1} \, (-1)^{nk_2} \, (-t_2)_{nk_2}}{ (c_1)_{m} \, (c_2)_n \, m! \, n!} \, \frac{\partial}{\partial x} x^{b + m + n} \, y^n\nonumber\\
	& = \sum_{m, n \geq 0} \frac{(a)_{m+n} \, (b)_{m + n} \, (-1)^{m k_1} \, (-t_1)_{mk_1} \, (-1)^{nk_2} \, (-t_2)_{nk_2}}{ (c_1)_{m} \, (c_2)_n \, m! \, n!} \, (b + m + n) \, x^{b + m + n - 1} \, y^n\nonumber\\
	& = x^{b - 1} \, b \sum_{m, n \geq 0} \frac{(a)_{m+n} \, (b + 1)_{m + n}\, (-1)^{m k_1} \, (-t_1)_{mk_1} \, (-1)^{nk_2} \, (-t_2)_{nk_2}}{ (c_1)_{m} \, (c_2)_n \, m! \, n!} \,  x^{m} \, y^n\nonumber\\
	& = x^{b - 1} \, b \mathcal{F}^{(1)}_4(a, b+ 1; c_1, c_2; t_1, t_2, k_1, k_2, x, y). 
\end{align}
Performing the same operation twice, we get
\begin{align}
	& \left(\frac{\partial}{\partial x}\right)^2 \left[x^{b + 1} \, \mathcal{F}^{(1)}_4(a, b; c_1, c_2; t_1, t_2, k_1, k_2, x, y)\right]\nonumber\\
	& = x^{b - 1} \, (b)_2 \, \mathcal{F}^{(1)}_4(a, b + 2; c_1, c_2; t_1, t_2, k_1, k_2, x, y). 
\end{align}
Inductively, this gives
\begin{align}
	& \left(\frac{\partial}{\partial x}\right)^r \left[x^{b + r - 1} \, \mathcal{F}^{(1)}_4(a, b; c_1, c_2; t_1, t_2, k_1, k_2, x, y)\right]\nonumber\\
	& = x^{b - 1} \, (b)_r \, \mathcal{F}^{(1)}_4(a, b + r; c_1, c_2; t_1, t_2, k_1, k_2, x, y).
\end{align}
It completes the proof. The other results can also be proved in a similar manner.
\end{proof}
\section{Recursion Formulae}
\begin{theorem} 
	The following recursion formulas hold for the discrete Appell function $\mathcal{F}^{(1)}_4$:
\begin{align}
& \mathcal{F}^{(1)}_4 (a + s, b; c_1, c_2 ; t_1, t_2, k_1, k_2, x, y) \nonumber\\
& = \mathcal{F}^{(1)}_4 (a, b; c_1, c_2; t_1, t_2, k_1, k_2, x, y) \nonumber\\
& \quad + \frac{(-1)^{k_1} \, (-t_1)_{k_1} \, b \, x}{c_1}  \sum_{r = 1}^{s} \mathcal{F}^{(1)}_4 (a + r, b + 1; c_1 + 1, c_2; t_1 - k_1, t_2, k_1, k_2, x, y)\nonumber\\
& \quad + \frac{(-1)^{k_2} \, (-t_2)_{k_2} \, b \, y}{c_2}  \sum_{r = 1}^{s} \mathcal{F}^{(1)}_4 (a + r, b + 1; c_1, c_2 + 1; t_1, t_2 - k_2, k_1, k_2, x, y);\label{e6.1}\\
 & \mathcal{F}^{(1)}_4 (a - s, b; c_1, c_2 ; t_1, t_2, k_1, k_2, x, y) \nonumber\\
& = \mathcal{F}^{(1)}_4 (a, b; c_1, c_2 ; t_1, t_2, k_1, k_2, x, y) \nonumber\\
& \quad - \frac{(-1)^{k_1}\, (-t_1)_{k_1} \, b \, x}{c_1}  \sum_{r = 0}^{s - 1} \mathcal{F}^{(1)}_4 (a - r, b + 1; c_1 + 1, c_2; t_1 - k_1, t_2, k_1, k_2, x, y) \nonumber\\
& \quad - \frac{(-1)^{k_2} \, (-t_2)_{k_2} \, b \, y}{c_2}  \sum_{r = 0}^{s - 1} \mathcal{F}^{(1)}_4 (a - r, b + 1; c_1, c_2 + 1; t_1, t_2 - k_2, k_1, k_2, x, y);\\
& \mathcal{F}^{(1)}_4 (a, b + s; c_1, c_2 ; t_1, t_2, k_1, k_2, x, y) \nonumber\\
& = \mathcal{F}^{(1)}_4 (a, b; c_1, c_2; t_1, t_2, k_1, k_2, x, y) \nonumber\\
& \quad + \frac{(-1)^{k_1} \, (-t_1)_{k_1} \, a \, x}{c_1}  \sum_{r = 1}^{s} \mathcal{F}^{(1)}_4 (a + 1, b + r; c_1 + 1, c_2; t_1 - k_1, t_2, k_1, k_2, x, y) \nonumber\\
& \quad + \frac{(-1)^{k_2} \, (-t_2)_{k_2} \, a \, y}{c_2}  \sum_{r = 1}^{s} \mathcal{F}^{(1)}_4 (a + 1, b + r; c_1, c_2 + 1; t_1, t_2 - k_2, k_1, k_2, x, y);\\
& \mathcal{F}^{(1)}_4 (a, b - s; c_1, c_2 ; t_1, t_2, k_1, k_2, x, y) \nonumber\\
& = \mathcal{F}^{(1)}_4 (a, b; c_1, c_2 ; t_1, t_2, k_1, k_2, x, y) \nonumber\\
& \quad - \frac{(-1)^{k_1} \, (-t_1)_{k_1} \, a \, x}{c_1}  \sum_{r = 0}^{s - 1} \mathcal{F}^{(1)}_4 (a + 1, b - r; c_1 + 1, c_2; t_1 - k_1, t_2, k_1, k_2, x, y) \nonumber\\
& \quad - \frac{(-1)^{k_1} \, (-t_2)_{k_2} \, a \, y}{c_2}  \sum_{r = 0}^{s - 1} \mathcal{F}^{(1)}_4 (a + 1, b - r; c_1, c_2 + 1; t_1, t_2 - k_2, k_1, k_2, x, y);\\
& \mathcal{F}^{(1)}_4 (a, b; c_1 - s, c_2; t_1, t_2, k_1, k_2, x, y) \nonumber\\
& = \mathcal{F}^{(1)}_4 (a, b; c_1, c_2 ; t_1, t_2, k_1, k_2, x, y) \nonumber\\
& \quad + (-1)^{k_1} \, (-t_1)_{k_1} \, a \, b \, x  \sum_{r = 1}^{s} \frac{\mathcal{F}^{(1)}_4 (a + 1, b + 1; c_1 + 2 - r, c_2; t_1 - k_1, t_2, k_1, k_2, x, y)}{(c_1 - r) \, (c_1 - r + 1)} \nonumber\\
& \quad + (-1)^{k_2} \, (-t_2)_{k_2} \, a \, b \, y  \sum_{r = 1}^{s} \frac{\mathcal{F}^{(1)}_4 (a + 1, b + 1; c_1 + 2 - r, c_2; t_1, t_2 - k_2, k_1, k_2, x, y)}{(c_1 - r) \, (c_1 - r + 1)}.
\end{align}
\end{theorem}
\begin{proof}
To prove the formula \eqref{e6.1}, we begin with
\begin{align}
 &\mathcal{F}^{(1)}_4 (a, b; c_1, c_2 ; t_1, t_2, k_1, k_2, x, y) + \frac{(-1)^{k_1} \, (-t_1)_{k_1} \, b \, x}{c_1} \nonumber\\
& \quad \times  \mathcal{F}^{(1)}_4 (a + 1, b + 1; c_1 + 1, c_2; t_1 - k_1, t_2, k_1, k_2, x, y) + \frac{(-1)^{k_2} \, (-t_2)_{k_2} \, b \, y}{c_2} \nonumber\\
& \quad \times  \mathcal{F}^{(1)}_4 (a + 1, b + 1; c_1, c_2 + 1; t_1, t_2 - k_2, k_1, k_2, x, y)\nonumber\\
& = \sum_{m,n \geq 0}  \, \frac{(a)_{m+n} \, (b)_{m + n} \,  (-1)^{m k_1} \, (-t_1)_{mk_1} \, (-1)^{n k_2} \, (-t_2)_{nk_2}}{ (c_1)_{m} \, (c_2)_n \, m! \, n!} \ x^m \, y^n + \frac{(-1)^{k_1} \, (-t_1)_{k_1} \, b \, x}{c_1} \nonumber\\
& \quad \times \sum_{m,n \geq 0}  \, \frac{(a + 1)_{m+n} \, (b + 1)_{m + n} \,  (-1)^{m k_1} \, (-t_1 + k_1)_{mk_1} \, (-1)^{nk_2} \, (-t_2)_{nk_2}}{ (c_1 + 1)_{m} \, (c_2)_n \, m! \, n!} \ x^m \, y^n + \frac{(-1)^{k_2} \, (-t_2)_{k_2} \, b \, y}{c_2} \nonumber\\
& \quad \times \sum_{m, n \geq 0}  \, \frac{(a + 1)_{m+n} \, (b + 1)_{m + n} \, (-1)^{mk_1} \, (-t_1)_{mk_1} \, (-1)^{nk_2} \, (-t_2 + k_2)_{nk_2}}{(c_1)_m \, (c_2 + 1)_{n} \, m! \, n!} \ x^m \, y^n\nonumber\\
& = \sum_{m,n \geq 0}  \, \frac{(a)_{m+n} \, (b)_{m + n} \, (-1)^{m k_1} \, (-t_1)_{mk_1} \, (-1)^{nk_2} \, (-t_2)_{nk_2}}{(c_1)_{m} \, (c_2)_n \, m! \, n!} \ x^m \, y^n \nonumber\\
& \quad + \sum_{m, n \geq 0} \frac{m}{a} \, \frac{(a)_{m+n} \, (b)_{m + n} \, (-1)^{mk_1} \, (-t_1)_{mk_1} \, (-1)^{nk_2} \, (-t_2)_{nk_2}}{(c_1)_{m} \, (c_2)_n \, m! \, n!} \ x^m \, y^n \nonumber\\
& \quad + \sum_{m, n \geq 0}  \frac{n}{a} \, \frac{(a)_{m+n} \, (b)_{m + n} \, (-1)^{mk_1} \, (-t_1)_{mk_1} \, (-1)^{nk_2} \, (-t_2)_{nk_2}}{ (c_1)_{m} \, (c_2)_n \, m! \, n!} \ x^m \, y^n\nonumber\\
& =  \sum_{m,n \geq 0}  \frac{a + m + n}{a} \, \frac{(a)_{m+n} \, (b)_{m + n} \, (-1)^{mk_1} \, (-t_1)_{mk_1} \, (-1)^{nk_2} \, (-t_2)_{nk_2}}{ (c_1)_{m} \, (c_2)_n \, m! \, n!} \ x^m \, y^n\nonumber\\
& = \mathcal{F}^{(1)}_4 (a + 1, b; c_1, c_2; t_1, t_2, k_1, k_2, x, y). 
\end{align}
Again
\begin{align}
	&\mathcal{F}^{(1)}_4 (a, b; c_1, c_2 ; t_1, t_2, k_1, k_2, x, y) + \frac{(-1)^{k_1} \, (-t_1)_{k_1} \, b \, x}{c_1} \nonumber\\
	& \quad \times  [\mathcal{F}^{(1)}_4 (a + 1, b + 1; c_1 + 1, c_2; t_1 - k_1, t_2, k_1, k_2, x, y) \nonumber\\
	& \quad  + \mathcal{F}^{(1)}_4 (a + 2, b + 1; c_1 + 1, c_2; t_1 - k_1, t_2, k_1, k_2, x, y)]  + \frac{(-1)^{k_2} \, (-t_2)_{k_2} \, b \, y}{c_2} \nonumber\\
	& \quad \times [\mathcal{F}^{(1)}_4 (a + 1, b + 1; c_1, c_2 + 1; t_1, t_2 - k_2, k_1, k_2, x, y)\nonumber\\
	& \quad   + \mathcal{F}^{(1)}_4 (a + 2, b + 1; c_1, c_2 + 1; t_1, t_2 - k_2, k_1, k_2, x, y)]\nonumber\\
	& =  \sum_{m,n \geq 0} \frac{(a)_{m+n} \, (b)_{m + n} \, (-1)^{m k_1} \, (-t_1)_{mk_1} \, (-1)^{nk_2} \,  (-t_2)_{nk_2}}{ (c_1)_{m} \, (c_2)_n \, m! \, n!} \ x^m \, y^n  + \frac{(-1)^{k_1} \, (-t_1)_{k_1} \, b \, x}{c_1} \nonumber\\
	& \quad \times \sum_{m,n \geq 0} (a + 2)_{m + n - 1} (2a + m + n + 2) \nonumber\\
	& \quad \times \frac{(b + 1)_{m + n}  \, (-1)^{m k_1} \, (-t_1 + k_1)_{mk_1} \, (-1)^{nk_2} \, (-t_2)_{nk_2}}{ (c_1 + 1)_{m} \, (c_2)_n \, m! \, n!} \ x^m \, y^n\nonumber\\
	& \quad + \frac{(-1)^{k_2} \, (-t_2)_{k_2} \, b \, y}{c_2} \sum_{m,n \geq 0} (a + 2)_{m + n - 1} (2a + m + n + 2)\nonumber\\
	& \quad \times \frac{ (b + 1)_{m + n} \,  (-1)^{m k_1} \, (-t_1)_{mk_1} \, (-1)^{nk_2} \, (-t_2 + k_2)_{nk_2}}{(c_1)_m \, (c_2 + 1)_{n} \, m! \, n!} \ x^m \, y^n\nonumber\\
	& = \sum_{m, n \ge 0} \frac{a \, (a + 1) + (2a + m + n + 1) (m + n)}{a \, (a + 1)}\nonumber\\
	& \quad \times \frac{(a)_{m + n} \, (b)_{m + n} \, (-1)^{mk_1} \, (-t_1)_{mk_1} \, (-1)^{nk_2} \, (-t_2)_{nk_2}}{ (c_1)_{m} \, (c_2)_n \, m! \, n!} \ x^m \, y^n\nonumber\\
	& = \frac{(a + 2)_{m + n} \, (b)_{m + n} \, (-1)^{mk_1} \, (-t_1)_{mk_1} \, (-1)^{nk_2} \, (-t_2)_{nk_2}}{ (c_1)_{m} \, (c_2)_n \, m! \, n!} \ x^m \, y^n\nonumber\\
	& = \mathcal{F}^{(1)}_4 (a + 2, b; c_1, c_2; t_1, t_2, k_1, k_2, x, y).
\end{align}
We can generalize the result and get the formula \eqref{e6.1}. Hence the proof is completed. The other formulae can be verified in the same manner. 
\end{proof}
Besides, we have several difference and differential recursion formulae. We start our findings with the differential recursion formulae first. To obtain, we list simple differential relations as
\begin{align}
&	a \, \mathcal{F}^{(1)}_4 (a + 1)  = (a + \theta + \phi) \, \mathcal{F}^{(1)}_4;\\
& (a + \theta + \phi - 1) \, \mathcal{F}^{(1)}_4 (a - 1)  =	(a - 1) \, \mathcal{F}^{(1)}_4;\\
&	b \, \mathcal{F}^{(1)}_4 (b + 1)  = (b + \theta + \phi) \, \mathcal{F}^{(1)}_4;\\
&(b + \theta + \phi - 1) \, \mathcal{F}^{(1)}_4 (b - 1)  =	(b - 1) \, \mathcal{F}^{(1)}_4;\\
&	(c_1 - 1) \, \mathcal{F}^{(1)}_4 (c_1 - 1)  = (c_1 + \theta - 1) \, \mathcal{F}^{(1)}_4;\\
& (c_1 + \theta) \, \mathcal{F}^{(1)}_4 (c_1 + 1)  =	c_1 \, \mathcal{F}^{(1)}_4;\\
&	(c_2 - 1) \, \mathcal{F}^{(1)}_4 (c_2 - 1)  = (c_2 + \phi - 1) \, \mathcal{F}^{(1)}_4;\\
& (c_2 + \phi) \, \mathcal{F}^{(1)}_4 (c_2 + 1)  =	c_2 \, \mathcal{F}^{(1)}_4.
\end{align}
On combining any two of the above relations, we get a  list of differential recursion formulas as follows:
\begin{align}
& a \, (a - 1) \, \mathcal{F}^{(1)}_4 (a + 1) - (a + \theta + \phi) \, (a + \theta + \phi - 1) \mathcal{F}^{(1)}_4 (a - 1) = 0;\\
& a \, (b - 1) \, \mathcal{F}^{(1)}_4 (a + 1) - (a + \theta + \phi) \, (b + \theta + \phi - 1) \mathcal{F}^{(1)}_4 (b - 1) = 0;\\
& a \, c_1 \, \mathcal{F}^{(1)}_4 (a + 1) - (a + \theta + \phi) \, (c_1 + \theta) \mathcal{F}^{(1)}_4 (c_1 + 1) = 0;\\
& a \, c_2 \, \mathcal{F}^{(1)}_4 (a + 1) - (a + \theta + \phi) \, (c_2 + \phi) \mathcal{F}^{(1)}_4 (c_2 + 1) = 0;\\
& a \, (b + \theta + \phi) \, \mathcal{F}^{(1)}_4 (a + 1) - b \, (a + \theta + \phi) \, \mathcal{F}^{(1)}_4 (b + 1) = 0;\\
	& a \, (c_1 + \theta - 1) \, \mathcal{F}^{(1)}_4 (a + 1) - (c_1 - 1) \, (a + \theta + \phi) \, \mathcal{F}^{(1)}_4 (c_1 - 1) = 0;\\
	& a \, (c_2 + \phi - 1) \, \mathcal{F}^{(1)}_4 (a + 1) - (c_2 - 1) \, (a + \theta + \phi) \, \mathcal{F}^{(1)}_4 (c_2 - 1) = 0;\\
&(a + \theta + \phi - 1) \, (b + \theta + \phi) \, \mathcal{F}^{(1)}_4 (a - 1) - b \, (a - 1)  \, \mathcal{F}^{(1)}_4 (b + 1) = 0;\\
 &	(a + \theta + \phi - 1) \, (c_1 + \theta - 1) \, \mathcal{F}^{(1)}_4 (a - 1) - (c_1 - 1) \, (a - 1)  \, \mathcal{F}^{(1)}_4 (c_1 - 1) = 0;\\
 & a \, (c_2 + \phi - 1) \, \mathcal{F}^{(1)}_4 (a + 1) - (c_2 - 1) \, (a + \theta + \phi) \, \mathcal{F}^{(1)}_4 (c_2 - 1) = 0;\\
 & (b - 1) \,	(a + \theta + \phi - 1) \,  \mathcal{F}^{(1)}_4 (a - 1) -  (a - 1)  \, (b + \theta + \phi - 1) \mathcal{F}^{(1)}_4 (b - 1) = 0;\\
 & c_1 \,	(a + \theta + \phi - 1) \,  \mathcal{F}^{(1)}_4 (a - 1) -  (a - 1)  \, (c_1 + \theta) \mathcal{F}^{(1)}_4 (c_1 + 1) = 0;\\
 & c_2 \,	(a + \theta + \phi - 1) \,  \mathcal{F}^{(1)}_4 (a - 1) -  (a - 1)  \, (c_1 + \phi) \mathcal{F}^{(1)}_4 (c_2 + 1) = 0;\\
 & b \,	(b - 1) \,  \mathcal{F}^{(1)}_4 (b + 1) -  (b + \theta + \phi)  \, (b + \theta + \phi - 1) \mathcal{F}^{(1)}_4 (b - 1) = 0;\\
  & b \,	(c_1 + \theta - 1) \,  \mathcal{F}^{(1)}_4 (b + 1) -  (c_1 - 1)  \, (b + \theta + \phi) \mathcal{F}^{(1)}_4 (c_1 - 1) = 0;\\
  & b \,	(c_2 + \phi - 1) \,  \mathcal{F}^{(1)}_4 (b + 1) -  (c_2 - 1)  \, (b + \theta + \phi) \mathcal{F}^{(1)}_4 (c_2 - 1) = 0;\\
  & b \,	c_1 \,  \mathcal{F}^{(1)}_4 (b + 1) -  (c_1 + \theta)  \, (b + \theta + \phi) \mathcal{F}^{(1)}_4 (c_1 + 1) = 0;\\
  & b \,	c_2 \,  \mathcal{F}^{(1)}_4 (b + 1) -  (c_2 + \phi)  \, (b + \theta + \phi) \mathcal{F}^{(1)}_4 (c_2 + 1) = 0;\\
  & (b + \theta + \phi - 1) \,	(c_1 + \theta - 1) \,  \mathcal{F}^{(1)}_4 (b - 1) -  (c_1 - 1)  \, (b - 1) \mathcal{F}^{(1)}_4 (c_1 - 1) = 0;\\
  & (b + \theta + \phi - 1) \,	(c_2 + \phi - 1) \,  \mathcal{F}^{(1)}_4 (b - 1) -  (c_2 - 1)  \, (b - 1) \mathcal{F}^{(1)}_4 (c_2 - 1) = 0;\\
  & 	c_1 \, (b + \theta + \phi - 1) \, \mathcal{F}^{(1)}_4 (b - 1) - (b - 1)  (c_1 + \theta)  \,  \mathcal{F}^{(1)}_4 (c_1 + 1) = 0;\\
  & 	c_2 \, (b + \theta + \phi - 1) \, \mathcal{F}^{(1)}_4 (b - 1) - (b - 1)  (c_2 + \phi)  \,  \mathcal{F}^{(1)}_4 (c_2 + 1) = 0;\\
   & 	c_1 \, (c_1 - 1) \, \mathcal{F}^{(1)}_4 (c_1 - 1) - (c_1 + \theta - 1)  (c_1 + \theta)  \,  \mathcal{F}^{(1)}_4 (c_1 + 1) = 0;\\
   & 	 (c_1 - 1) \, (c_2 + \phi - 1) \, \mathcal{F}^{(1)}_4 (c_1 - 1) - (c_2 - 1)  \, (c_1 + \phi - 1)    \mathcal{F}^{(1)}_4 (c_2 - 1) = 0;\\
   & 	c_2 \, (c_1 - 1) \, \mathcal{F}^{(1)}_4 (c_1 - 1) - (c_1 + \phi - 1)  (c_2 + \phi)  \,  \mathcal{F}^{(1)}_4 (c_2 + 1) = 0;\\
   & 	 (c_1 + \theta) \, (c_2 + \phi - 1) \, \mathcal{F}^{(1)}_4 (c_1 + 1) - c_1 \, (c_2 - 1)  \,  \mathcal{F}^{(1)}_4 (c_2 - 1) = 0;\\
   & 	c_2 \, (c_1 + \theta) \, \mathcal{F}^{(1)}_4 (c_1 + 1) - c_1 \, (c_2 + \phi)  \,  \mathcal{F}^{(1)}_4 (c_2 + 1) = 0;\\
   & 	c_2 \, (c_2 - 1) \, \mathcal{F}^{(1)}_4 (c_2 - 1) - (c_2 + \phi - 1)  (c_2 + \phi)  \,  \mathcal{F}^{(1)}_4 (c_2 + 1) = 0.
 \end{align}
Similarly using  the difference relations, 
\begin{align}
	&	a \, \mathcal{F}^{(1)}_4 (a + 1)  = \left(a + \frac{1}{k_1} \Theta_{t_1} + \frac{1}{k_2} \Theta_{t_2}\right) \, \mathcal{F}^{(1)}_4;\\
	& \left(a + \frac{1}{k_1} \Theta_{t_1} + \frac{1}{k_2} \Theta_{t_2} - 1\right) \, \mathcal{F}^{(1)}_4 (a - 1)  =	(a - 1) \, \mathcal{F}^{(1)}_4;\\
	&	b \, \mathcal{F}^{(1)}_4 (b + 1)  = \left(b + \frac{1}{k_1} \Theta_{t_1} + \frac{1}{k_2} \Theta_{t_2}\right) \, \mathcal{F}^{(1)}_4;\\
	&\left(b + \frac{1}{k_1} \Theta_{t_1} + \frac{1}{k_2} \Theta_{t_2} - 1\right) \, \mathcal{F}^{(1)}_4 (b - 1)  =	(b - 1) \, \mathcal{F}^{(1)}_4;\\
	&	(c_1 - 1) \, \mathcal{F}^{(1)}_4 (c_1 - 1)  = \left(c_1 + \frac{1}{k_1} \Theta_{t_1} - 1\right) \, \mathcal{F}^{(1)}_4;\\
	& \left(c_1 + \frac{1}{k_1} \Theta_{t_1}\right) \, \mathcal{F}^{(1)}_4 (c_1 + 1)  =	c_1 \, \mathcal{F}^{(1)}_4;\\
	&	(c_2 - 1) \, \mathcal{F}^{(1)}_4 (c_2 - 1)  = \left(c_2 + \frac{1}{k_2} \Theta_{t_2} - 1\right) \, \mathcal{F}^{(1)}_4;\\
	& \left(c_2 + \frac{1}{k_2} \Theta_{t_2}\right) \, \mathcal{F}^{(1)}_4 (c_2 + 1)  =	c_2 \, \mathcal{F}^{(1)}_4,
\end{align}
one will get the following difference recursion relations obeyed by discrete Appell function $\mathcal{F}^{(1)}_4$. The proofs are straightforward.
\begin{align}
	& a \, (a - 1) \, \mathcal{F}^{(1)}_4 (a + 1)\nonumber\\
	& \quad - \left(a + \frac{1}{k_1} \Theta_{t_1} + \frac{1}{k_2} \Theta_{t_2}\right) \, \left(a + \frac{1}{k_1} \Theta_{t_1} + \frac{1}{k_2} \Theta_{t_2} - 1\right) \mathcal{F}^{(1)}_4 (a - 1) = 0;\\
	& a \, (b - 1) \, \mathcal{F}^{(1)}_4 (a + 1) \nonumber\\
	& \quad - \left(a + \frac{1}{k_1} \Theta_{t_1} + \frac{1}{k_2} \Theta_{t_2} \right) \, \left(b + \frac{1}{k_1} \Theta_{t_1} + \frac{1}{k_2} \Theta_{t_2} - 1\right) \mathcal{F}^{(1)}_4 (b - 1) = 0;\\
	& a \, c_1 \, \mathcal{F}^{(1)}_4 (a + 1)  - \left(a + \frac{1}{k_1} \Theta_{t_1} + \frac{1}{k_2} \Theta_{t_2}\right) \, \left(c_1 + \frac{1}{k_1} \Theta_{t_1} \right) \mathcal{F}^{(1)}_4 (c_1 + 1) = 0;\\
& a \, c_2 \, \mathcal{F}^{(1)}_4 (a + 1)  - \left(a + \frac{1}{k_1} \Theta_{t_1} + \frac{1}{k_2} \Theta_{t_2}\right) \, \left(c_2 + \frac{1}{k_2} \Theta_{t_2} \right) \mathcal{F}^{(1)}_4 (c_2 + 1) = 0;\\
	& a \, \left(b + \frac{1}{k_1} \Theta_{t_1} + \frac{1}{k_2} \Theta_{t_2}\right) \, \mathcal{F}^{(1)}_4 (a + 1) - b \, \left(a + \frac{1}{k_1} \Theta_{t_1} + \frac{1}{k_2} \Theta_{t_2}\right) \, \mathcal{F}^{(1)}_4 (b + 1) = 0;\\
	& a \, \left(c_1 + \frac{1}{k_1} \Theta_{t_1} - 1\right) \, \mathcal{F}^{(1)}_4 (a + 1)  - (c_1 - 1) \, \left(a + \frac{1}{k_1} \Theta_{t_1} + \frac{1}{k_2} \Theta_{t_2}\right) \, \mathcal{F}^{(1)}_4 (c_1 - 1) = 0;\\
	& a \, \left(c_2 + \frac{1}{k_2} \Theta_{t_2} - 1\right) \, \mathcal{F}^{(1)}_4 (a + 1)  - (c_2 - 1) \, \left(a + \frac{1}{k_1} \Theta_{t_1} + \frac{1}{k_2} \Theta_{t_2}\right) \, \mathcal{F}^{(1)}_4 (c_2 - 1) = 0;\\
	&\left(a + \frac{1}{k_1} \Theta_{t_1} + \frac{1}{k_2} \Theta_{t_2} - 1\right) \, \left(b + \frac{1}{k_1} \Theta_{t_1} + \frac{1}{k_2} \Theta_{t_2} \right) \, \mathcal{F}^{(1)}_4 (a - 1)\nonumber\\
	& \quad - b \, (a - 1)  \, \mathcal{F}^{(1)}_4 (b + 1) = 0;\\
	&	\left(a + \frac{1}{k_1} \Theta_{t_1} + \frac{1}{k_2} \Theta_{t_2} - 1\right) \, \left(c_1 + \frac{1}{k_1} \Theta_{t_1} - 1\right) \, \mathcal{F}^{(1)}_4 (a - 1)\nonumber\\
	& \quad  - (c_1 - 1) \, (a - 1)  \, \mathcal{F}^{(1)}_4 (c_1 - 1) = 0;\\
	&	\left(a + \frac{1}{k_1} \Theta_{t_1} + \frac{1}{k_2} \Theta_{t_2} - 1\right) \, \left(c_2 + \frac{1}{k_2} \Theta_{t_2} - 1\right) \, \mathcal{F}^{(1)}_4 (a - 1)\nonumber\\
	& \quad  - (c_2 - 1) \, (a - 1)  \, \mathcal{F}^{(1)}_4 (c_2 - 1) = 0;\\
	& (b - 1) \,	\left(a + \frac{1}{k_1} \Theta_{t_1} + \frac{1}{k_2} \Theta_{t_2} - 1\right) \,  \mathcal{F}^{(1)}_4 (a - 1) \nonumber\\
	& \quad -  (a - 1)  \, \left(b + \frac{1}{k_1} \Theta_{t_1} + \frac{1}{k_2} \Theta_{t_2} - 1\right) \mathcal{F}^{(1)}_4 (b - 1) = 0;\\
	& c_1 \,	\left(a + \frac{1}{k_1} \Theta_{t_1} + \frac{1}{k_2} \Theta_{t_2} - 1\right) \,  \mathcal{F}^{(1)}_4 (a - 1) -  (a - 1)  \, \left(c_1 + \frac{1}{k_1} \Theta_{t_1}\right) \mathcal{F}^{(1)}_4 (c_1 + 1) = 0;\\
	& c_2 \,	\left(a + \frac{1}{k_1} \Theta_{t_1} + \frac{1}{k_2} \Theta_{t_2} - 1\right) \,  \mathcal{F}^{(1)}_4 (a - 1)  -  (a - 1)  \, \left(c_2 + \frac{1}{k_2} \Theta_{t_2}\right) \mathcal{F}^{(1)}_4 (c_2 + 1) = 0;\\
	& b \,	(b - 1) \,  \mathcal{F}^{(1)}_4 (b + 1) \nonumber\\
	& \quad -  \left(b + \frac{1}{k_1} \Theta_{t_1} + \frac{1}{k_2} \Theta_{t_2} \right)  \, \left(b + \frac{1}{k_1} \Theta_{t_1} + \frac{1}{k_2} \Theta_{t_2}  - 1\right) \mathcal{F}^{(1)}_4 (b - 1) = 0;\\
	& b \,	\left(c_1 + \frac{1}{k_1} \Theta_{t_1} -  1\right) \,  \mathcal{F}^{(1)}_4 (b + 1) -  (c_1 - 1)  \, \left(b + \frac{1}{k_1} \Theta_{t_1} + \frac{1}{k_2} \Theta_{t_2} \right) \mathcal{F}^{(1)}_4 (c_1 - 1) = 0;\\
	& b \,	\left(c_2 + \frac{1}{k_2} \Theta_{t_2} -  1\right) \,  \mathcal{F}^{(1)}_4 (b + 1) -  (c_2 - 1)  \, \left(b + \frac{1}{k_1} \Theta_{t_1} + \frac{1}{k_2} \Theta_{t_2} \right) \mathcal{F}^{(1)}_4 (c_2 - 1) = 0;\\
	& b \,	c_1 \,  \mathcal{F}^{(1)}_4 (b + 1) -  \left(c_1 + \frac{1}{k_1} \Theta_{t_1}\right)  \, \left(b + \frac{1}{k_1} \Theta_{t_1} + \frac{1}{k_2} \Theta_{t_2}\right) \mathcal{F}^{(1)}_4 (c_1 + 1) = 0;\\
	& b \,	c_2 \,  \mathcal{F}^{(1)}_4 (b + 1) -  \left(c_2 + \frac{1}{k_2} \Theta_{t_2}\right)  \, \left(b + \frac{1}{k_1} \Theta_{t_1} + \frac{1}{k_2} \Theta_{t_2}\right) \mathcal{F}^{(1)}_4 (c_2 + 1) = 0;\\
	& \left(b + \frac{1}{k_1} \Theta_{t_1} + \frac{1}{k_2} \Theta_{t_2} - 1\right) \,	\left(c_1 + \frac{1}{k_1} \Theta_{t_1} - 1\right) \,  \mathcal{F}^{(1)}_4 (b - 1)\nonumber\\
	& \quad -  (c_1 - 1)  \, (b - 1) \mathcal{F}^{(1)}_4 (c_1 - 1) = 0;\\
	& \left(b + \frac{1}{k_1} \Theta_{t_1} + \frac{1}{k_2} \Theta_{t_2} - 1\right) \,	\left(c_2 + \frac{1}{k_2} \Theta_{t_2} - 1\right) \,  \mathcal{F}^{(1)}_4 (b - 1)\nonumber\\
	& \quad -  (c_2 - 1)  \, (b - 1) \mathcal{F}^{(1)}_4 (c_2 - 1) = 0;\\
	& 	c_1 \, \left(b + \frac{1}{k_1} \Theta_{t_1} + \frac{1}{k_2} \Theta_{t_2} - 1\right) \, \mathcal{F}^{(1)}_4 (b - 1) - (b - 1)  \left(c_1 + \frac{1}{k_1} \Theta_{t_1}\right)  \,  \mathcal{F}^{(1)}_4 (c_1 + 1) = 0;\\
	& 	c_2 \, \left(b + \frac{1}{k_1} \Theta_{t_1} + \frac{1}{k_2} \Theta_{t_2} - 1\right) \, \mathcal{F}^{(1)}_4 (b - 1) - (b - 1)  \left(c_2 + \frac{1}{k_2} \Theta_{t_2}\right)  \,  \mathcal{F}^{(1)}_4 (c_2 + 1) = 0;\\
	& 	c_1 \, (c_1 - 1) \, \mathcal{F}^{(1)}_4 (c_1 - 1) - \left(c_1 + \frac{1}{k_1} \Theta_{t_1} - 1\right)  \left(c_1 + \frac{1}{k_1} \Theta_{t_1}\right)  \,  \mathcal{F}^{(1)}_4 (c_1 + 1) = 0;\\
	& 	 (c_1 - 1) \, \left(c_2 + \frac{1}{k_2} \Theta_{t_2} - 1\right) \, \mathcal{F}^{(1)}_4 (c_1 - 1) - (c_2 - 1)  \, \left(c_1 + \frac{1}{k_1} \Theta_{t_1} - 1\right)    \mathcal{F}^{(1)}_4 (c_2 - 1) = 0;\\
	& 	c_2 \, (c_1 - 1) \, \mathcal{F}^{(1)}_4 (c_1 - 1) - \left(c_1 + \frac{1}{k_1} \Theta_{t_1} - 1\right)  \left(c_2 + \frac{1}{k_2} \Theta_{t_2}\right) \,  \mathcal{F}^{(1)}_4 (c_2 + 1) = 0;\\
	& 	 \left(c_1 + \frac{1}{k_1} \Theta_{t_1}\right) \, \left(c_2 + \frac{1}{k_2} \Theta_{t_2} - 1\right) \, \mathcal{F}^{(1)}_4 (c_1 + 1) - c_1 \, (c_2 - 1)  \,  \mathcal{F}^{(1)}_4 (c_2 - 1) = 0;\\
	& 	c_2 \, \left(c_1 + \frac{1}{k_1} \Theta_{t_1} \right) \, \mathcal{F}^{(1)}_4 (c_1 + 1) - c_1 \, \left(c_2 + \frac{1}{k_2} \Theta_{t_2}\right)  \,  \mathcal{F}^{(1)}_4 (c_2 + 1) = 0;\\
	& 	c_2 \, (c_2 - 1) \, \mathcal{F}^{(1)}_4 (c_2 - 1) - \left(c_2 + \frac{1}{k_2} \Theta_{t_2} - 1\right)  \left(c_2 + \frac{1}{k_2} \Theta_{t_2}\right)  \,  \mathcal{F}^{(1)}_4 (c_2 + 1) = 0.
\end{align}

\section{Discrete  Appell function $F_4^{(2)}$}
The second discrete form of Appell hypergeometric  function ${F}_4$  defined in \eqref{3.2} converges absolutely when $\sqrt{\vert x\vert} + \sqrt{\vert y\vert} < 1$. Also for $k = 0$ and $k = 1$, the discrete Appell function $F_4^{(2)}$ reduces into classical Appell function and Kamp\'e de F\'eriet functions as follows: 
\begin{align}
	\mathcal{F}^{(2)}_4(a, b; c_1, c_2; t, 0, x, y) 
	& = F_4 (a, b; c_1, c_2; x, y).
\end{align}
\begin{align}
\mathcal{F}^{(2)}_4(a, b; c_1, c_2; t, 1, x, y)
	& = F_{{0}:{1}, {1}} ^{{3}:{0}, {0}}\left(\begin{array}{ccc}
		a, b, -t: & -  & -\\
		-: & c_1, & c_2 
	\end{array}; x, \ y\right).
\end{align}
The following difference-differential equations are satisfied by  $\mathcal{F}^{(2)}_4$:
\begin{align}
	&	\left[\theta \left(\theta + c_1 - 1\right)  - (-1)^k \, (-t)_k \, x \, \rho_t^k \,  \left(\frac{1}{k} \Theta_{t} + a\right) \, \left(\frac{1}{k} \Theta_{t}  + b\right) \right] \mathcal{F}^{(2)}_4 = 0;\nonumber
	\\[5pt]
	&	\left[\phi \, \left(\phi + c_2 - 1\right) - (-1)^k \, (-t)_k \, y \, \rho_t^k \,  \left(\frac{1}{k} \Theta_{t} + a\right) \, \left(\frac{1}{k} \Theta  + b \right) \right] \mathcal{F}^{(2)}_4 = 0.
\end{align}	    
\subsection{Differential and difference formulae}
\begin{align}
	& (\theta)^r \mathcal{F}^{(2)}_4(a, b; c_1, c_2; t, k, x, y) \nonumber\\
	& = \frac{(-1)^{rk} \, (a)_r \, (b)_r \, (-t)_{rk} \, x^r}{(c_1)_r} \, \mathcal{F}^{(2)}_4(a + r, b + r; c_1 + r, c_2; t - rk, k, x, y);\\
	& (\phi)^r \mathcal{F}^{(2)}_4 (a, b; c_1, c_2; t, k, x, y) \nonumber\\
	& = \frac{(-1)^{rk} \, (a)_r \, (b)_r \, (-t)_{rk} \, y^r}{(c_2)_r} \, \mathcal{F}^{(2)}_4 (a + r, b + r; c_1, c_2 + r; t - rk, k, x, y).
\end{align}
The particular case of these formulae lead to the difference and differential formulae for  Kamp\'e de F\'eriet type hypergeometric function as:
\begin{align}
	& (\theta)^r \mathcal{F}^{(2)}_4 (a, b; c_1, c_2; t, 1, x, y) \nonumber\\
	& = \frac{(-1)^{r} \, (a)_r \, (b)_r \, (-t)_{r} \, x^r}{(c_1)_r} \, \mathcal{F}^{(2)}_4 (a + r, b + r; c_1 + r, c_2; t - r, 1, x, y);\\
	& (\phi)^r \mathcal{F}^{(2)}_4 (a, b; c_1, c_2; t, 1, x, y) \nonumber\\
	& = \frac{(-1)^{r} \, (a)_r \, (b)_r \, (-t)_{r} \, y^r}{(c_2)_r} \, \mathcal{F}^{(2)}_4(a + r, b + r; c_1, c_2 + r; t - r, 1, x, y).
\end{align}
Besides, some other differential formulas for $\mathcal{F}^{(2)}_4$ are as follows  
\begin{align}
	& \left(\frac{\partial}{\partial x}\right)^r  \left[x^{b + r - 1} \mathcal{F}^{(2)}_4 (a, b; c_1, c_2; t, k, x, x y)\right]\nonumber\\
	&\qquad = x^{b - 1} \, (b)_r \mathcal{F}^{(2)}_4(a, b + r; c_1, c_2; t, k, x, x \, y);\\
	& \left(\frac{\partial}{\partial y}\right)^r [y^{b + r - 1} \mathcal{F}^{(2)}_4(a, b; c_1, c_2; t, k, x \, y, y)]\nonumber\\
	&\qquad  = y^{b - 1} \, (b)_r \mathcal{F}^{(2)}_4 (a, b + r; c_1, c_2; t, k, x \, y, y);\\
	& \left(\frac{\partial}{\partial x}\right)^r [x^{a + r - 1} \mathcal{F}^{(2)}_4 (a, b; c_1, c_2; t, k, x, xy)]\nonumber\\
	&\qquad  = x^{a - 1} \, (a)_r \mathcal{F}^{(2)}_4 (a + r, b; c_1, c_2; t, k, x, xy);\\
	& \left(\frac{\partial}{\partial y}\right)^r [y^{a + r - 1} \mathcal{F}^{(2)}_4 (a, b; c_1, c_2; t, k, xy, y)]\nonumber\\
	&\qquad  = y^{a - 1} \, (a)_r \mathcal{F}^{(2)}_4(a + r, b; c_1, c_2; t, k, xy, y);\\
	& \left(\frac{\partial}{\partial x}\right)^r [x^{c_1 - 1} \mathcal{F}^{(2)}_4(a, b; c_1, c_2; t, k, x, y)]\nonumber\\
	&\qquad  = (-1)^r \, (1 - c_1)_r \, x^{c_1 - r - 1} \mathcal{F}^{(2)}_4(a, b; c_1 - r, c_2; t, k, x, y);\\
	& \left(\frac{\partial}{\partial y}\right)^r [y^{c_2 - 1} \mathcal{F}^{(2)}_4(a, b; c_1, c_2; t, k, x, y)]\nonumber\\
	&\qquad  = (-1)^r \, y^{c_2 - r - 1} \, (1 - c_2)_r \mathcal{F}^{(2)}_4 (a, b; c_1, c_2 - r; t, k, x, y).
\end{align}

\subsection{Recursion Formulae}
The following recursion formulas hold for the discrete Appell function $\mathcal{F}^{(2)}_4$ :
\begin{align}
	& \mathcal{F}^{(2)}_4 (a + s, b; c_1, c_2 ; t, k, x, y) \nonumber\\
	& = \mathcal{F}^{(2)}_4 (a, b; c_1, c_2; t, k, x, y)\nonumber\\
	& \quad + \frac{(-1)^k \, (-t)_k \, b \, x}{c_1}  \sum_{r = 1}^{s} \mathcal{F}^{(2)}_4 (a + r, b + 1; c_1 + 1, c_2; t - k, k, x, y)\nonumber\\
	& \quad + \frac{(-1)^k \, (-t)_k \, b \, y}{c_2}  \sum_{r = 1}^{s} \mathcal{F}^{(2)}_4 (a + r, b + 1; c_1, c_2 + 1; t - k, k, x, y);\\
	& \mathcal{F}^{(2)}_4 (a - s, b; c_1, c_2 ; t, k, x, y) \nonumber\\
	& = \mathcal{F}^{(2)}_4 (a, b; c_1, c_2 ; t, k, x, y) \nonumber\\
	& \quad - \frac{(-1)^k \, (-t)_k \, b \, x}{c_1}  \sum_{r = 0}^{s - 1} \mathcal{F}^{(2)}_4 (a - r, b + 1; c_1 + 1, c_2; t - k, k, x, y)\nonumber\\
	& \quad - \frac{(-1)^k \, (-t)_k \, b \, y}{c_2}  \sum_{r = 0}^{s - 1} \mathcal{F}^{(2)}_4 (a - r, b + 1; c_1, c_2 + 1; t - k, k, x, y);\\
	& \mathcal{F}^{(2)}_4 (a, b + s; c_1, c_2 ; t, k, x, y) \nonumber\\
	& = \mathcal{F}^{(2)}_4 (a, b; c_1, c_2; t, k, x, y) \nonumber\\
	& \quad + \frac{(-1)^k \, (-t)_k \, a \, x}{c_1} \sum_{r = 1}^{s} \mathcal{F}^{(2)}_4 (a + 1, b + r; c_1 + 1, c_2; t - k, k, x, y) \nonumber\\
	& \quad + \frac{(-1)^k \, (-t)_k \, a \, y}{c_2}  \sum_{r = 1}^{s} \mathcal{F}^{(2)}_4 (a + 1, b + r; c_1, c_2 + 1; t - k, k, x, y);\\
	& \mathcal{F}^{(2)}_4 (a, b - s; c_1, c_2 ; t, k, x, y) \nonumber\\
	& = \mathcal{F}^{(2)}_4 (a, b; c_1, c_2 ; t, k, x, y) \nonumber\\
	& \quad - \frac{(-1)^k \, (-t)_k \, a \, x}{c_1}  \sum_{r = 0}^{s - 1} \mathcal{F}^{(2)}_4 (a + 1, b - r; c_1 + 1, c_2; t_1 - k, t_2, k, x, y) \nonumber\\
	& \quad - \frac{(-1)^k \, (-t)_k \, a \, y}{c_2}  \sum_{r = 0}^{s - 1} \mathcal{F}^{(2)}_4 (a + 1, b - r; c_1, c_2 + 1; t - k, k, x, y);\\
	& \mathcal{F}^{(2)}_4 (a, b; c_1 - s, c_2; t, k, x, y) \nonumber\\
	& = \mathcal{F}^{(2)}_4 (a, b; c_1, c_2 ; t, k, x, y)\nonumber\\
	& \quad + (-1)^k \, (-t)_k \, a \, b \, x  \sum_{r = 1}^{s} \frac{\mathcal{F}^{(2)}_4 (a + 1, b + 1; c_1 + 2 - r, c_2; t - k,  k, x, y)}{(c_1 - r) \, (c_1 - r + 1)}\nonumber\\
	& \quad + (-1)^k \, (-t)_k \, a \, b \, y \sum_{r = 1}^{s} \frac{\mathcal{F}^{(2)}_4 (a + 1, b + 1; c_1 + 2 - r, c_2; t - k, k, x, y)}{(c_1 - r) \, (c_1 - r + 1)}.
\end{align}
Besides, we have several other difference and differential recursion formulae satisfied by $\mathcal{F}^{(2)}_4$. First, we give the differential recursion formulae obtained using the simple relations
\begin{align}
	&	a \, \mathcal{F}^{(2)}_4 (a + 1)  = (a + \theta + \phi) \, \mathcal{F}^{(2)}_4;\label{6.19}\\
	& (a + \theta + \phi - 1) \, \mathcal{F}^{(2)}_4 (a - 1)  =	(a - 1) \, \mathcal{F}^{(2)}_4;\\
	&	b \, \mathcal{F}^{(2)}_4 (b + 1)  = (b + \theta + \phi) \, \mathcal{F}^{(2)}_4;\\
	&(b + \theta + \phi - 1) \, \mathcal{F}^{(2)}_4 (b - 1)  =	(b - 1) \, \mathcal{F}^{(2)}_4;\\
	&	(c_1 - 1) \, \mathcal{F}^{(2)}_4 (c_1 - 1)  = (c_1 + \theta - 1) \, \mathcal{F}^{(2)}_4;\\
	& (c_1 + \theta) \, \mathcal{F}^{(2)}_4 (c_1 + 1)  =	c_1 \, \mathcal{F}^{(2)}_4;\\
	&	(c_2 - 1) \, \mathcal{F}^{(2)}_4 (c_2 - 1)  = (c_2 + \phi - 1) \, \mathcal{F}^{(2)}_4;\\
	& (c_2 + \phi) \, \mathcal{F}^{(2)}_4 (c_2 + 1)  =	c_2 \, \mathcal{F}^{(2)}_4.\label{6.26}
\end{align}
On combining any two of the  relations \eqref{6.19}- \eqref{6.26}, we get a first order or second order differential recursion relations satisfied by $\mathcal{F}^{(2)}_4$. We produce here a list of such relations. The proofs are straightforward.
\begin{align}
	& a \, (a - 1) \, \mathcal{F}^{(2)}_4 (a + 1) - (a + \theta + \phi) \, (a + \theta + \phi - 1) \mathcal{F}^{(2)}_4 (a - 1) = 0;\\
	& a \, (b - 1) \, \mathcal{F}^{(2)}_4 (a + 1) - (a + \theta + \phi) \, (b + \theta + \phi - 1) \mathcal{F}^{(2)}_4 (b - 1) = 0;\\
	& a \, c_1 \, \mathcal{F}^{(2)}_4 (a + 1) - (a + \theta + \phi) \, (c_1 + \theta) \mathcal{F}^{(2)}_4 (c_1 + 1) = 0;\\
	& a \, c_2 \, \mathcal{F}^{(2)}_4 (a + 1) - (a + \theta + \phi) \, (c_2 + \phi) \mathcal{F}^{(2)}_4 (c_2 + 1) = 0;\\
	& a \, (b + \theta + \phi) \, \mathcal{F}^{(2)}_4 (a + 1) - b \, (a + \theta + \phi) \, \mathcal{F}^{(2)}_4 (b + 1) = 0;\\
	& a \, (c_1 + \theta - 1) \, \mathcal{F}^{(2)}_4 (a + 1) - (c_1 - 1) \, (a + \theta + \phi) \, \mathcal{F}^{(2)}_4 (c_1 - 1) = 0;\\
	& a \, (c_2 + \phi - 1) \, \mathcal{F}^{(2)}_4 (a + 1) - (c_2 - 1) \, (a + \theta + \phi) \, \mathcal{F}^{(2)}_4 (c_2 - 1) = 0;\\
	&(a + \theta + \phi - 1) \, (b + \theta + \phi) \, \mathcal{F}^{(2)}_4 (a - 1) - b \, (a - 1)  \, \mathcal{F}^{(2)}_4 (b + 1) = 0;\\
	&	(a + \theta + \phi - 1) \, (c_1 + \theta - 1) \, \mathcal{F}^{(2)}_4 (a - 1) - (c_1 - 1) \, (a - 1)  \, \mathcal{F}^{(2)}_4 (c_1 - 1) = 0;\\
	& a \, (c_2 + \phi - 1) \, \mathcal{F}^{(2)}_4 (a + 1) - (c_2 - 1) \, (a + \theta + \phi) \, \mathcal{F}^{(2)}_4 (c_2 - 1) = 0;\\
	& (b - 1) \,	(a + \theta + \phi - 1) \,  \mathcal{F}^{(2)}_4 (a - 1) -  (a - 1)  \, (b + \theta + \phi - 1) \mathcal{F}^{(2)}_4 (b - 1) = 0;\\
	& c_1 \,	(a + \theta + \phi - 1) \,  \mathcal{F}^{(2)}_4 (a - 1) -  (a - 1)  \, (c_1 + \theta) \mathcal{F}^{(2)}_4 (c_1 + 1) = 0;\\
	& c_2 \,	(a + \theta + \phi - 1) \,  \mathcal{F}^{(2)}_4 (a - 1) -  (a - 1)  \, (c_1 + \phi) \mathcal{F}^{(2)}_4 (c_2 + 1) = 0;\\
	& b \,	(b - 1) \,  \mathcal{F}^{(2)}_4 (b + 1) -  (b + \theta + \phi)  \, (b + \theta + \phi - 1) \mathcal{F}^{(2)}_4 (b - 1) = 0;\\
	& b \,	(c_1 + \theta - 1) \,  \mathcal{F}^{(2)}_4 (b + 1) -  (c_1 - 1)  \, (b + \theta + \phi) \mathcal{F}^{(2)}_4 (c_1 - 1) = 0;\\
	& b \,	(c_2 + \phi - 1) \,  \mathcal{F}^{(2)}_4 (b + 1) -  (c_2 - 1)  \, (b + \theta + \phi) \mathcal{F}^{(2)}_4 (c_2 - 1) = 0;\\
	& b \,	c_1 \,  \mathcal{F}^{(2)}_4 (b + 1) -  (c_1 + \theta)  \, (b + \theta + \phi) \mathcal{F}^{(2)}_4 (c_1 + 1) = 0;\\
	& b \,	c_2 \,  \mathcal{F}^{(2)}_4 (b + 1) -  (c_2 + \phi)  \, (b + \theta + \phi) \mathcal{F}^{(2)}_4 (c_2 + 1) = 0;\\
	& (b + \theta + \phi - 1) \,	(c_1 + \theta - 1) \,  \mathcal{F}^{(2)}_4 (b - 1) -  (c_1 - 1)  \, (b - 1) \mathcal{F}^{(2)}_4 (c_1 - 1) = 0;\\
	& (b + \theta + \phi - 1) \,	(c_2 + \phi - 1) \,  \mathcal{F}^{(2)}_4 (b - 1) -  (c_2 - 1)  \, (b - 1) \mathcal{F}^{(2)}_4 (c_2 - 1) = 0;\\
	& 	c_1 \, (b + \theta + \phi - 1) \, \mathcal{F}^{(2)}_4 (b - 1) - (b - 1)  (c_1 + \theta)  \,  \mathcal{F}^{(2)}_4 (c_1 + 1) = 0;\\
	& 	c_2 \, (b + \theta + \phi - 1) \, \mathcal{F}^{(2)}_4 (b - 1) - (b - 1)  (c_2 + \phi)  \,  \mathcal{F}^{(2)}_4 (c_2 + 1) = 0;\\
	& 	c_1 \, (c_1 - 1) \, \mathcal{F}^{(2)}_4 (c_1 - 1) - (c_1 + \theta - 1)  (c_1 + \theta)  \,  \mathcal{F}^{(2)}_4 (c_1 + 1) = 0;\\
	& 	 (c_1 - 1) \, (c_2 + \phi - 1) \, \mathcal{F}^{(2)}_4 (c_1 - 1) - (c_2 - 1)  \, (c_1 + \phi - 1)    \mathcal{F}^{(2)}_4 (c_2 - 1) = 0;\\
	& 	c_2 \, (c_1 - 1) \, \mathcal{F}^{(2)}_4 (c_1 - 1) - (c_1 + \phi - 1)  (c_2 + \phi)  \,  \mathcal{F}^{(2)}_4 (c_2 + 1) = 0;\\
	& 	 (c_1 + \theta) \, (c_2 + \phi - 1) \, \mathcal{F}^{(2)}_4 (c_1 + 1) - c_1 \, (c_2 - 1)  \,  \mathcal{F}^{(2)}_4 (c_2 - 1) = 0;\\
	& 	c_2 \, (c_1 + \theta) \, \mathcal{F}^{(2)}_4 (c_1 + 1) - c_1 \, (c_2 + \phi)  \,  \mathcal{F}^{(2)}_4 (c_2 + 1) = 0;\\
	& 	c_2 \, (c_2 - 1) \, \mathcal{F}^{(2)}_4 (c_2 - 1) - (c_2 + \phi - 1)  (c_2 + \phi)  \,  \mathcal{F}^{(2)}_4 (c_2 + 1) = 0.
\end{align}
Next, we list the difference-differential recursion relations satisfied by discrete Appell function $\mathcal{F}^{(2)}_4$ in terms of the following differential and difference relations: 
\begin{align}
	&	a \, \mathcal{F}^{(2)}_4 (a + 1)  = \left(a + \frac{1}{k} \Theta_{t} \right) \, \mathcal{F}^{(2)}_4;\\
	& \left(a + \frac{1}{k} \Theta_{t}  - 1\right) \, \mathcal{F}^{(2)}_4 (a - 1)  =	(a - 1) \, \mathcal{F}^{(2)}_4;\\
	&	b \, \mathcal{F}^{(2)}_4 (b + 1)  = \left(b + \frac{1}{k} \Theta_{t}\right) \, \mathcal{F}^{(2)}_4;\\
	&\left(b + \frac{1}{k} \Theta_{t} - 1\right) \, \mathcal{F}^{(2)}_4 (b - 1)  =	(b - 1) \, \mathcal{F}^{(2)}_4;\\
	&	(c_1 - 1) \, \mathcal{F}^{(2)}_4 (c_1 - 1)  = \left(c_1 + \theta - 1\right) \, \mathcal{F}^{(2)}_4;\\
	& \left(c_1 + \theta\right) \, \mathcal{F}^{(2)}_4 (c_1 + 1)  =	c_1 \, \mathcal{F}^{(2)}_4;\\
	&	(c_2 - 1) \, \mathcal{F}^{(2)}_4 (c_2 - 1)  = \left(c_2 + \phi - 1\right) \, \mathcal{F}^{(2)}_4;\\
	& \left(c_2 + \phi\right) \, \mathcal{F}^{(2)}_4 (c_2 + 1)  =	c_2 \, \mathcal{F}^{(2)}_4.
\end{align}
The difference-differential recursion relations obeyed by discrete Appell function $\mathcal{F}^{(2)}_4$ are listed below: 
\begin{align}
	& a \, (a - 1) \, \mathcal{F}^{(2)}_4 (a + 1) - \left(a + \frac{1}{k} \Theta_{t}\right) \, \left(a + \frac{1}{k} \Theta_{t} - 1\right) \mathcal{F}^{(2)}_4 (a - 1) = 0;\\
	& a \, (b - 1) \, \mathcal{F}^{(2)}_4 (a + 1)  - \left(a + \frac{1}{k} \Theta_{t}\right) \, \left(b + \frac{1}{k} \Theta_{t} - 1\right) \mathcal{F}^{(2)}_4 (b - 1) = 0;\\
	& a \, c_1 \, \mathcal{F}^{(2)}_4 (a + 1)   - \left(a + \frac{1}{k} \Theta_{t}\right) \, \left(c_1 + \theta \right) \mathcal{F}^{(2)}_4 (c_1 + 1) = 0;\\
	& a \, c_2 \, \mathcal{F}^{(2)}_4 (a + 1)  - \left(a + \frac{1}{k} \Theta_{t}\right) \, \left(c_2 + \phi \right) \mathcal{F}^{(2)}_4 (c_2 + 1) = 0;\\
	& a \, \left(b + \frac{1}{k} \Theta_{t}\right) \, \mathcal{F}^{(2)}_4 (a + 1) - b \, \left(a + \frac{1}{k} \Theta_{t}\right) \, \mathcal{F}^{(2)}_4 (b + 1) = 0;\\
	& a \, \left(c_1 + \theta - 1\right) \, \mathcal{F}^{(2)}_4 (a + 1)  - (c_1 - 1) \, \left(a +\frac{1}{k} \Theta_{t}\right) \, \mathcal{F}^{(2)}_4 (c_1 - 1) = 0;\\
	& a \, \left(c_2 + \phi - 1\right) \, \mathcal{F}^{(2)}_4 (a + 1)  - (c_2 - 1) \, \left(a + \frac{1}{k} \Theta_{t}\right) \, \mathcal{F}^{(2)}_4 (c_2 - 1) = 0;\\
	&\left(a + \frac{1}{k} \Theta_{t} - 1\right) \, \left(b + \frac{1}{k} \Theta_{t} \right) \, \mathcal{F}^{(2)}_4 (a - 1) - b \, (a - 1)  \, \mathcal{F}^{(2)}_4 (b + 1) = 0;\\
	&	\left(a + \frac{1}{k} \Theta_{t} - 1\right) \, \left(c_1 + \theta - 1\right) \, \mathcal{F}^{(2)}_4 (a - 1)   - (c_1 - 1) \, (a - 1)  \, \mathcal{F}^{(2)}_4 (c_1 - 1) = 0;\\
	&	\left(a + \frac{1}{k} \Theta_{t} - 1\right) \, \left(c_2 + \phi - 1\right) \, \mathcal{F}^{(2)}_4 (a - 1) - (c_2 - 1) \, (a - 1)  \, \mathcal{F}^{(2)}_4 (c_2 - 1) = 0;\\
	& (b - 1) \,	\left(a + \frac{1}{k} \Theta_{t} - 1\right) \,  \mathcal{F}^{(2)}_4 (a - 1) -  (a - 1)  \, \left(b + \frac{1}{k} \Theta_{t} - 1\right) \mathcal{F}^{(2)}_4 (b - 1) = 0;\\
	& c_1 \,	\left(a + \frac{1}{k} \Theta_{t} - 1\right) \,  \mathcal{F}^{(2)}_4 (a - 1) -  (a - 1)  \, \left(c_1 + \theta\right) \mathcal{F}^{(2)}_4 (c_1 + 1) = 0;\\
	& c_2 \,	\left(a + \frac{1}{k} \Theta_{t} - 1\right) \,  \mathcal{F}^{(2)}_4 (a - 1) -  (a - 1)  \, \left(c_2 + \phi\right) \mathcal{F}^{(2)}_4 (c_2 + 1) = 0;\\
	& b \,	(b - 1) \,  \mathcal{F}^{(2)}_4 (b + 1) -  \left(b + \frac{1}{k} \Theta_{t} \right)  \, \left(b + \frac{1}{k} \Theta_{t}  - 1\right) \mathcal{F}^{(2)}_4 (b - 1) = 0;\\
	& b \,	\left(c_1 + \theta -  1\right) \,  \mathcal{F}^{(2)}_4 (b + 1) -  (c_1 - 1)  \, \left(b + \frac{1}{k} \Theta_{t}\right) \mathcal{F}^{(2)}_4 (c_1 - 1) = 0;\\
	& b \,	\left(c_2 + \phi -  1\right) \,  \mathcal{F}^{(2)}_4 (b + 1) -  (c_2 - 1)  \, \left(b + \frac{1}{k} \Theta_{t} \right) \mathcal{F}^{(2)}_4 (c_2 - 1) = 0;\\
	& b \,	c_1 \,  \mathcal{F}^{(2)}_4 (b + 1) -  \left(c_1 + \theta\right)  \, \left(b + \frac{1}{k} \Theta_{t}\right) \mathcal{F}^{(2)}_4 (c_1 + 1) = 0;\\
	& b \,	c_2 \,  \mathcal{F}^{(2)}_4 (b + 1) -  \left(c_2 + \phi\right)  \, \left(b + \frac{1}{k} \Theta_{t}\right) \mathcal{F}^{(2)}_4 (c_2 + 1) = 0;\\
	& \left(b + \frac{1}{k} \Theta_{t} - 1\right) \,	\left(c_1 + \theta - 1\right) \,  \mathcal{F}^{(2)}_4 (b - 1) -  (c_1 - 1)  \, (b - 1) \mathcal{F}^{(2)}_4 (c_1 - 1) = 0;\\
	& \left(b + \frac{1}{k} \Theta_{t} - 1\right) \,	\left(c_2 + \phi - 1\right) \,  \mathcal{F}^{(2)}_4 (b - 1) -  (c_2 - 1)  \, (b - 1) \mathcal{F}^{(2)}_4 (c_2 - 1) = 0;\\
	& 	c_1 \, \left(b + \frac{1}{k} \Theta_{t} - 1\right) \, \mathcal{F}^{(2)}_4 (b - 1) - (b - 1)  \left(c_1 + \theta\right)  \,  \mathcal{F}^{(2)}_4 (c_1 + 1) = 0;\\
	& 	c_2 \, \left(b + \frac{1}{k} \Theta_{t} - 1\right) \, \mathcal{F}^{(2)}_4 (b - 1) - (b - 1)  \left(c_2 + \phi\right)  \,  \mathcal{F}^{(2)}_4 (c_2 + 1) = 0.
\end{align}


\end{document}